\documentclass{article}
\usepackage{amsmath,amsfonts,euscript,amssymb,amsthm,float}
\usepackage[utf8]{inputenc}
\usepackage[english]{babel}
\usepackage[
backend=biber,
style=numeric,
sorting=ynt]{biblatex}
\newcommand{\field}[1]{\mathbb{#1}}
\newcommand{\R}{\field{R}}

\newcommand{\N}{\field{N}}

\newcommand{\textoverline}[1]{$\overline{\mbox{#1}}$}

\newtheorem{theorem}{Theorem}
\newtheorem{lemma}{Lemma}[section]
\newtheorem{cor}{Corollary}[section]
\theoremstyle{definition}
\newtheorem{defn}{Definition}[section]
\newtheorem{exmp}{Example}[section]
\newtheorem{rmk}{Remark}[section]

\title{\vspace{-20pt}Density of Collatz Trajectories}
\author{J. Cooper Faile\\The University of North Carolina at Chapel Hill\thanks{I am a recent graduate who completed this work under the guidance of Professor Idris Assani.}}

\begin{document}
\maketitle
\vspace{-20pt}
\begin{abstract}
In this paper I present a new method of studying the densities of the Collatz trajectories generated by a set $S \subset \mathbb{N}$. This method is used to furnish an alternative proof that $d(\{y \in \mathbb{N} : \exists k \text{ where } T^k(y) < cy\}) = 1$ for all $c > 0$. Finally, I briefly discuss how the ideas presented in this paper could be used to improve the result that $d(\{y \in \mathbb{N} : \exists k \text{ where } T^k(y) < y^c\}) = 1$ for $c > \log_4 3$.
\end{abstract}

\section{Introduction}

Define $T\colon \N\to \N$ by 
$$ T(n) = \begin{cases} \frac{n}{2} & 2|n \\ \frac{3n+1}{2} & \text{otherwise} \end{cases} $$
The Collatz conjecture states that for every $n\in \N$ there exists $k\in \N$ such that $T^k(n) = 1$. 
This conjecture is equivalent to the inductive statement that for each $n\in \N + 1$ there exists $k\in \N$ such that $T^k(n) < n$.
This equivalent statement led to the study of the density of subsets $S\subset \N$ which satisfy, for all $n\in S$, $T^k(n) < f(n)$ for functions $f:\N\to\R$ (i.e. subsets of $\N$ for which the sequence $n,\ T(n),\dots,\ T^k(n),\ \dots$ contains a sufficiently “small” iterate.)
Terras showed that the natural density of $S$ was $1$ for $f(n) = n$, Venturini showed the same result for $f(n) = c_1n$, $c_1 > 0$, and Korek showed it for $f(n) = n^{c_2}$, where $c_2 > \log_4 3$ \cite{korec-94, venturini-89, terras-76}.

\section{Notation and Prior Results}

This paper requires the following notation from \cite{terras-76}. Given $k,\ m,\ y \in \N$ and $d\in \R$ we define
\begin{align*}
  X_k(y) &= \begin{cases}
              1 & \text{if } T^k(y) \text{ is odd} \\
              0 & \text{if } T^k(y) \text{ is even,}
          \end{cases} \\
  S_k(y) &= X_0(y) + X_1(y) + \dots + X_{k}(y) \\
  U(m,d) &= \left| \{y\in \N: 0 \leq y < 2^m\text{ and } S_m(y)\leq md \}\right|
\end{align*}

We also need this lemma from \cite{terras-76}
\begin{lemma}
  We have that for any $b\in \N$
  $$ U(m,d) = \left| \left\{ y: b\leq y < b+2^m\text{ and } S_m(y) < md \right\} \right| $$
  and for any $d > 1/2$ we have
  $$ \lim_{m\to\infty} \frac{U(m,d)}{2^m} = 1. $$
  \label{prelim:clt}
\end{lemma}
and another lemma, based on the proof in \cite{korec-94}.
\begin{lemma}
  For any $M \geq 3$ there exists a $d\in \R$ such that for all $m > M$, we have that $y \geq m2^m$ and $S_m(y) < md$ implies $T^m(y) < y/2$.
  \label{prelim:cond}
\end{lemma}
\begin{proof}
  Note that since $y \geq m2^m$ we have that $T^j(y) \geq y/2^j \geq m$ for all $0\leq j \leq m$. 
  Letting $k = S_m(y)$ we have the bound
  \begin{align*}
    T^m(y) &= y \cdot \frac{T(y)}{y}\cdot \left( \cdots \right)\cdot\frac{T^m(y)}{T^{m-1}(y)} < y\cdot \left( \frac{3m+1}{2m} \right)^{k}\cdot \frac{1}{2^{m-k}} \\
    &= y\cdot \frac{3^k}{2^m} \cdot \left( 1 + \frac{1}{3m} \right)^k \leq y\cdot \frac{3^k}{2^m}\cdot \left( 1 + \frac{1}{3m} \right)^m < y\cdot \frac{3^k}{2^{m-1}}.
  \end{align*}
  The value $y\cdot(3^k/2^{m-1}) < y/2$ whenever $3^k < 2^{m-2}$, which is equivalent to
  $$ S_m(y) = k < \frac{m-2}{\log_2 3}.$$
  Since $m > M \geq 3$ we can select $d$ such that $dm < (m-2)/\log_2 3$ for $m > M$. 
  hence any $d$ with $0 < d < (M-2)/(M\log_2 3)$ satisfies this lemma. 
\end{proof}

\section{Positive Density Preserving Functions}
As a reminder,
\begin{defn}
  Given $A\subset \N$ the \textbf{upper density} of $A$, denoted $\overline{d}(A)$, is
  $$ \overline{\lim_{n\to\infty}} \frac{|A\cap [0,n]|}{n}. $$
  The \textbf{lower density} of $A$, denoted $\underline{d}(A)$, is
  $$ \underline{\lim_{n\to\infty}} \frac{|A\cap [0,n]|}{n}. $$
  If $\overline{d}(A) = \underline{d}(A)$ we say $A$ has \textbf{density} $d(A) = \overline{d}(A)$. 
\end{defn}

To study the density of Collatz trajectories of $S\subset \N$ we will consider functions which maintain positive density. 
This property guarantees that no set that is sufficiently large (i.e. $\overline{d}(A) > 0$) has an image that is small (i.e. $\overline{d}(f(A)) = 0$.) 
This property can then be used to determine the density for certain sets.

\begin{defn}
  A function $H\colon 2^\N\to 2^\N$ is \textbf{positive upper density preserving} (which will be abbreviated p\textoverline{d}p) if $\overline{d}(A) > 0$ implies that $\overline{d}(H(A)) > 0$ for any $A\subset \N$. 
\end{defn}

\begin{exmp}
  $T$ is p\textoverline{d}p. For any set $A\subset \N$ with $\overline{d}(A) > 0$ let $A = O\cup E$ where $O$, $E$ are the odd and even elements of $A$, respectively. 
  We must have that one of $\overline{d}(O) $ or $\overline{d}(E)$ is greater than $\overline{d}(A)/2$.
  In these respective cases we have
  \begin{align*}
    \overline{d}(T(A)) \geq \overline{d}(T(E)) = 2\overline{d}(E) \geq \overline{d}(A) \\
    \overline{d}(T(A)) \geq \overline{d}(T(O)) = \frac{2}{3}\overline{d}(O) \geq \frac{1}{3} \overline{d}(A)
  \end{align*}
  which shows that $T$ is p\textoverline{d}p. 
\end{exmp}

\begin{lemma}
  If $f$ and $g$ is p\textoverline{d}p then $f\circ g$ and $f^k$ is p\textoverline{d}p for any $k\in \N$. 
  \label{sep:fg}
\end{lemma}
\begin{proof}
  If $\overline{d}(A) > 0$ then $\overline{d}(g(A)) > 0$, $\overline{d}( (f\circ g)(A)) > 0$ showing $f\circ g$ is p\textoverline{d}p. 
  By induction the same reasoning shows that $f^k$ is p\textoverline{d}p.
\end{proof}

Now, given an increasing function $f\colon \N\to \R$ we can define $H_f\colon 2^\N \to 2^\N$ by
$$ H_f(A) = \{T^k(n): n\in A \text{ and } T^k(n) < f(n)\}. $$
This map sends a subset $A\subset \N$ to the set containing (partial) trajectories of every $n\in A$.
If this map is p\textoverline{d}p for some $f$ then it has some nice properties. 
\begin{lemma}
  If $g(n) \leq f(n)$ for all $n\in\N$ then $H_g(A)\subset H_f(A)$ for all $A\subset \N$. 
  Additionally, if $H_g$ is p\textoverline{d}p then $H_f$ is p\textoverline{d}p. %
  \label{sep:comparison}
\end{lemma}
\begin{proof}
  If $m\in H_g(A)$ then there exist $k\in \N$ and $n \in A$ where $m = T^k(n) < g(n)$. 
  Since $g(n) < f(n)$ we must have $m\in H_f(A)$ and hence $H_g(A) \subset H_f(A)$.
  If $H_g$ is p\textoverline{d}p then for all $A\subset \N$ with positive upper density we have $\overline{d}(H_f(A))\geq\overline{d}(H_g(A)) > 0$ showing that $H_f$ is p\textoverline{d}p. 
\end{proof}

\begin{lemma}
  For all $A\subset \N$ and $f,g:\N\to\R$ we have $(H_f\circ H_g)(A) \subset H_{f\circ g}(A)$.
  \label{sep:composition}
\end{lemma}
\vspace{-25pt}
\begin{proof}
  Any element of $(H_f\circ H_g)(A)$ can be written as $T^{k+k'}(n)$ for $k,\ k',\ n\in\N$
    satisfying $T^k(n) < g(n)$ and $T^{k'+k}(n) < f(T^k(n))$ by the definitions of $H_f$ and $H_g$. 
  Since $f$ and $g$ are increasing functions, we have 
  $T^{k+k'}(n) < f(T^k(n)) < f(g(n))$ and hence $T^{k+k'}(n) \in H_{f\circ g}(A)$. 
  This establishes the inclusion. 
\end{proof}

\begin{cor}
  For all $k\in \N,\ A\subset \N,$ and $f:\N\to \R$ increasing we have $(H_f)^k(A) \subset H_{f^k}(A)$. 
  Furthermore, if $H_f$ is p\textoverline{d}p then so is $H_{f^k}$. 
  \label{sep:recurse}
\end{cor}
\begin{proof}
  $(H_f)^k(A) \subset H_{f^k}(A)$ follows directly from lemma \ref{sep:composition}. 
  If $H_f$ is p\textoverline{d}p then by lemma \ref{sep:fg} $(H_f)^k$ is p\textoverline{d}p as well.
  Combining this with the inclusion $(H_f)^k(A) \subset H_{f^k}(A)$ we have
  $\overline{d}((H_f)^k(A))  \leq \overline{d}(H_{f^k}(A))$ for all $A\subset \N$. 
  Hence $H_{f^k}$ is p\textoverline{d}p. 
\end{proof}

\begin{lemma}
  If $H_f$ is p\textoverline{d}p then the set
  $$ M_f = \{y\in \N: \exists k\in \N\text{ where } T^k(y) < f(y)\}$$
  has density 1.
  \label{sep:density}
\end{lemma}
\begin{proof}
  Let $C = \N\setminus M_f$. We have $T^k(n) \geq f(n)$ for every $k \in \N$ and $n\in C$.
  Hence, $H_f(C) = \varnothing$ which has upper density zero, implying $\overline{d}(C) = 0$ by contrapositive. 
  Furthermore, $\underline{d}(M_f) = 1 - \overline{d}(C) = 1$ establishing $M_f$ has density 1. 
\end{proof}

The following theorem will be used with the previous lemmas to furnish an alternative proof of the result in \cite{venturini-89}.
\begin{theorem}
  Let $ f:\N\to \R$ be $f(y) = y/2$. Then $H_f$ is p\textoverline{d}p.
\end{theorem}

\begin{proof}

  Let $A\subset \N$ be a set with $\overline{d}(A) > 0$. 
  From lemma \ref{prelim:cond} there exists $d > 1/2$ such that for $m$ sufficiently large all $y \geq m2^m$ we have $T^m(y) < y/2 $ when $S_m(y) < md$. 
  Select $m\in \N$ such that $U(m,d)/2^m > 1-\overline{d}(A)/4$ and it is large enough for the above implication. 
  We can do this, as $U(m,d)/2^m \to 1$ as $m\to\infty$ for all $d > 1/2$ due to lemma \ref{prelim:clt}. 
  Define the set
  $$ L_{m,d} = \{y\in \N: S_m(y) < md\}. $$

  By definition of $\overline{d}$ there exists a sequence $(a_i)_{i=1}^\infty$ where 
  $$ \frac{| A\cap [0,a_i] |}{a_i} \to \overline{d}(A). $$
  For any $a_i > m2^m$ let $b_i$ be the greatest integer for which $b_i2^m \leq a_i$. 
  Then the set $A\cap [m2^m, b_i2^m]$ satisfies 
  \begin{align*}
    |A\cap [m2^m, b_i2^m]| &= |A\cap [0, a_i]| - |A\cap [0,m2^m)| - |A\cap (b_i2^m, a_i]| \\
    &\geq |A\cap [0, a_i]| - |\N\cap [0,m2^m)| - |\N\cap [b_i2^m, a_i]| \\
    &\geq |A\cap [0, a_i]| - m2^m - 2^m. \\
  \end{align*}
  Dividing by $a_i$ and taking the limsup we see
  $$ \limsup_{i\to\infty} \frac{|A\cap [m2^m, b_i2^m]|}{a_i} \geq \overline{d}(A).$$
  Reindexing the sequence $(b_i)_{i=1}^\infty$, we can assume
  $$ \frac{|A\cap [m2^m, b_i2^m]|}{a_i} \geq \frac{3\overline{d}(A)}{4}$$
  for all $i\in \N$. 
  Additionally, note that due to our choice of $m$ and lemma \ref{prelim:clt} we have that
  $$ L_{m,d}^c\cap [m2^m, b_i2^m] = (b_i-m)\cdot\left(2^m - U(m,d)\right) \leq \frac{\overline{d}(A)}{4}(b_i-m)2^m.$$

  The set $A\cap [m2^m, b_i2^m] \cap L_{m,d}$ is a subset of $A$ for which $T^m(y) < y/2$ for of all it elements (due to the choice of $d$ and lemma \ref{prelim:cond}.) 
  Using the above bounds, the size of this set is at least
  \begin{align*}
    |A\cap [m2^m, b_i2^m] \cap L_{m,d}| &\geq |A\cap [m2^m,b_i2^m]| - |L_{m,d}^c\cap [m2^m,b_i2^m]|\\
    &\geq \frac{3\overline{d}(A)}{4}\cdot a_i - \frac{\overline{d}(A)}{4}(b_i-m)2^m \geq \frac{\overline{d}(A)}{2}\cdot a_i.
  \end{align*}
  
  Using this bound and that $|T^m(B)| \geq |B|/2^m$ for any $B\subset \N$ we have the following bound on the size of $H_f$ of this set:
  \begin{align*}
    |H_f(A\cap [m2^m, b_i2^m])| &\geq |\{T^m(y): y\in A\cap [m2^m, b_i2^m] \text{ and }S_m(y) < md\}| \\
    &= |T^m(A\cap [m2^m, b_i2^m] \cap L_{m,d})| \\
    &\geq \frac{1}{2^m} \left| A\cap [m2^m, b_i2^m] \cap L_{m,d}\right| \geq \frac{a_i\overline{d}(A)}{2^{m+1}} 
  \end{align*}
  Dividing this inequality by $a_i$ and noting that $H_f(A) \cap [0,a_i] \supset H_f(A\cap [0,a_i])$ gives the density estimate
  \begin{align*}
    \frac{|H_f(A)\cap [0,a_i]|}{a_i} &\geq \frac{|H_f(A\cap [m2^m, b_i2^m])|}{a_i} \\
    &\geq \frac{a_i\overline{d}(A)}{2^{m+1}a_i} = \frac{\overline{d}(A)}{2^{m+1}} \\
  \end{align*}
  This shows that $\overline{d}(H_f(A)) \geq \overline{d}(A)/2^{m+1} > 0$ and hence that $H_f$ is p\textoverline{d}p. 
\end{proof}

\begin{cor}
  $H_f$ is p\textoverline{d}p for all $f:\N\to\R$ where $f(y) = c_1y$, $c_1 > 0$.
  \label{sep:linear-sep}
\end{cor}
\begin{proof}
  From theorem $1$ we know that $H_{g}$ is p\textoverline{d}p, where $g(y) = y/2$. 
  For any $c_1$ there exists $n$ such that $1/2^n < c_1$. 
  Then, by lemma \ref{sep:composition} we have $H_{g^n}$ is p\textoverline{d}p. 
  Since $g^n(y) = y/2^n < c_1y = f(y)$ we have $H_{g^n}(A)\subset H_f(A)$ for all $A\subset \N$ and hence by lemma \ref{sep:comparison} we have $H_f$ is p\textoverline{d}p.
\end{proof}

\begin{cor}
  The set $M_c = \{y: \exists k \text{ such that } T^k(y) < cy \}$ has density $1$ for all $c > 0$. 
\end{cor}
\begin{proof}
  This follows directly from lemma \ref{sep:density} and corollary \ref{sep:linear-sep}. 
\end{proof}

\begin{rmk}
  Suppose that $H_f$ is p\textoverline{d}p where $f:\N\to \R$ is the map $f(y) = y^c$ for a single $c < 1$.
  Applying lemma \ref{sep:composition} (as we did in corollary \ref{sep:linear-sep}) will give us $H_g$ is p\textoverline{d}p for $g(y) = y^{c'}$ for all $c' > 0$. 
  This directly gives a way to improve the result in \cite{korec-94}. 
  Proving that such a $H_f$ is p\textoverline{d}p would require a bound better than $|T^m(B)| \geq |B|/2^m$ (as was used in the proof of theorem 1.)

\end{rmk}

\section{Acknowledgements}
I would like to thank Professor Idris Assani for bringing these problems to my attention and for his assistance in creating this paper. 

\printbibliography
\end{document}